\documentclass[12pt,reqno]{amsart} 
\usepackage{amsmath,amssymb,rawfonts}
\usepackage{tikz}
\usepackage{mathrsfs}
\usepackage{scalerel}
\usepackage{tikz-cd}
\usetikzlibrary{arrows, matrix}
\usepackage{mathtools}

\newcommand{\be}{\begin{equation}}
\newcommand{\ee}{\end{equation}}

\newtheorem{theorem}{Theorem}[section]

\newtheorem{lemma}{Lemma}[section]

\newtheorem{remark}{Remark}[section]
\newtheorem{teor}[theorem]{Theorem}
\newtheorem{defi}[theorem]{Definition}
\newtheorem{prop}[theorem]{Proposition}
\newtheorem{coro}[theorem]{Corollary}
\newtheorem{notation}[theorem]{Notation}
 
\newtheorem{ex}[theorem]{Example}

\title{FILTERS AND COMPACTNESS ON  SMALL CATEGORIES AND LOCALES}                                  \author{Joaqu\'in Luna-Torres}
\thanks{Programa de Matem\'aticas, Universidad Distrital Francisco Jos\'e de Caldas,  Bogot\'a D. C., Colombia (retired professor)}

\email{jlunator@fundacionhaiko.org}

\begin{document} 

\begin{abstract}
In analogy with the classical theory of filters, for fi\-nite\-ly complete or small cat\-e\-go\-ries, we provide  the concepts of fil\-ter, $\mathfrak{G}$-neigh\-bor\-hood (short for ``Grothendieck-neigh\-bor\-hood")  and  cover-neigh\-bor\-hood of points of such categories,  with the goal of studying convergence, cluster point, closure of  sieves and compactness on objects  of that kind of categories. Finally, we study all these concepts in the category $\bf{Loc}$ of locales.
\end{abstract}

\maketitle 

\section{ Introduction}
Convergence theory offers a versatile and effective framework for some areas of mathematics. Let us start by saying a few words about the history of this concept.

Convergence theory was probably defined firstly by  Henri Cartan \cite{HC}.

The notion of a limit along a filter was defined in his work  in the maximum
generality, as a filter on an arbitrary set and the limit defined for any map from this set to a topological space. However, the attention of mathematicians in the following years was mostly focused on two special cases.

\begin{itemize}
\item In general topology the notion of limit of a filter on a topological space $X$ became one of the two basic tools used to describe the convergence in general topological spaces together with the notion of a net.
Some authors studied also the convergence of a sequence along a filter.  
 
\item The definition of the limit along a filter can be reformulated using ideals – the dual notion
 of filter -. This type of limit of sequences was introduced independently by P. Kostyrko et alt., \cite{KM}
 and F. Nuray and W. H. Ruckle \cite{NR} and studied under the name {\it{`` I-convergence"}}. The motivation for this direction of research was an effort to generalize some known results on statistical convergence. 
\end{itemize}
In category theory,  a sieve is a way of choosing arrows with a common codomain. It is a categorical analogue of a collection of open subsets of a fixed open set in topology. In a Grothendieck topology, certain sieves become categorical analogues of open covers in general topology. 

In this paper, we use the concept of sieve to build filters in categories and locales;
 we explore the relationship between filters and Grothendieck topologies,  definig the concept of $\mathfrak{G}$-convergence, in order to carry out the study of compactness.

The paper is organized as follows: We describe, in section $2$, the notion of sieve as in S. MacLane and I. Moerdijk \cite{MM} . In section $3$, we present the concepts of filters, filter base and we  study the lattice structure of  all filters on a category and we present the concept of ultrafilter; after, in section $4$, we establish a connection between filters and Grothendieck topologies in the same category. In section $5$ we introduce the concepts of systems of neighborhood, $\mathfrak{G}$-neighborhood  of a point (recall that a point is an arrow with domain a terminal object), cover-neighborhood, convergence, cluster point and closure of a sieve and some propositions about them . In section $6$ the notion of filter-preserving (or continuous) functor is presented; next, in section $7$ we use the convergence of ultrafilters in order to define compact objects in the categories in question. Finally, in section $8$, we study all the previous concepts in the category of locales.

\section{Theoretical Considerations}
In the first part of this paper, we will work within an ambient category $\mathscr{C} $ which is finitely complete, later we will do it on a small category.

From S. MacLane and I. Moerdijk \cite{MM}, Chapter III, we have the following:

Let $\mathscr{C}$ be a category and let $C$ be an object of $\mathscr{C}$. 
A sieve $\mathcal S$ on $C$ is a family of morphisms in $\mathscr{C}$, all with codomain $C$, such that $ f \in  \mathcal S \Longrightarrow f\circ g \in \mathcal S
$
whenever this composition makes sense; in other words, $\mathcal S$ is a right ideal. 

If $\mathcal S$ is a sieve on $\mathscr C$ and $h: D\rightarrow C$ is any arrow to $C$, then\linebreak  $h^{*}(\mathcal S) = \{g \mid cod(g) = D,\,\ h\circ g \in \mathcal S\}$  is a sieve on $D$.

The set  $Sieve(C)$, of all sieves on $C$, is a partially ordered set under inclusion. It is easy to see that the union or intersection of any family of sieves on $C$ is a sieve on $C$, so $Sieve(C)$ is a complete lattice.

\section{Filters on a category}
\begin{defi}
 A filter on a category $\mathscr{C}$ is a function $\mathfrak F$ which assigns to each object $C$ of $\mathscr{C}$ a  collection $\mathfrak F(C) $ of sieves on $C$, in such a way that
\begin{enumerate}
\item [($F_1$)] If $S \in \mathfrak F(C)$ and $R$ is a sieve on $C$ such that $ S \subseteq R$, then  $ R \in \mathfrak F(C)$;
\item [($F_2$)] every finite intersection of sieves of $\mathfrak F(C)$ belongs to $\mathfrak F(C)$;
\item [($F_3$)] if $S\in \mathfrak F(C)$, then  $h^{*}(S)\in \mathfrak F(D)$ for any arrow  $h: D\rightarrow C$;
\item [($F_4$)] the empty sieve is not in $\mathfrak F(C)$.
\end{enumerate}
\end{defi}
The pair $(C, \mathfrak F(C))$ will be called {\bf a filtered object}.
\begin{ex}
From the definition of a Grothendieck topology $J$ on a category $\mathscr{C}$  it follows that for each object $C$ of $\mathscr{C}$ and that
\begin{itemize}
\item for $S\in J(C)$ any larger sieve $R$ on $C$ is also a member of $J(C)$;
\item it is also clear that if $R; S\in J(C)$ then $R\cap S \in J(C)$; 
\item consequently some Grothendieck topologies produce  filters in the same category $\mathscr{C}$: they are exactly those for which   
$ R\cap S\neq\emptyset$ for all pairs $ R; S\in J(C)$  and such that  the empty sieve is not in $J(C)$.
\item Clearly, the {\bf trivial topology} on $\mathscr{C}$ is a filter we shall call it {\bf trivial filter}.
\item It is also clear that the {\bf atomic topology} on $\mathscr{C}$ (see \cite{MM})  is not a filter.
\end{itemize}
\end{ex}
\begin{remark}
According to the previous example, for any site $(\mathscr{C}, J)$ there is a dense sub-site, given by the full subcategory of $\mathscr{C}$ on the objects which are not covered by the empty sieve with the induced topology, whose topos of sheaves is equivalent to Sh(C; J). This is an immediate application of the Comparison Lemma (see \cite{PJ1} Theorem 2.2.3).
\end{remark}

\begin{defi}
A filter subbase on a category $\mathscr{C}$ is a function $\mathfrak S$ which assigns to each object $C$ of $\mathscr{C}$ a  collection $\mathfrak S(C) $ of sieves on $\mathscr{C}$,  in such a way that no finite subcollection of $\mathfrak S(C) $ has an empty intersection.
\end{defi}
An immediate consequence of this definition is
\begin{prop}
A sufficient condition that there should exist a filter $\mathfrak S^{'}$  on a category $\mathscr{C}$ greater than or equal to  a function $\mathfrak S$ (as above) is that $\mathfrak S$  should be a filter subbase on  $\mathscr{C}$.
\end{prop}
Observe that $\mathfrak S^{'}$ is the coarset filter greater than $\mathfrak S$.
\begin{defi}
A basis of a filter on a category $\mathscr{C}$ is a function $\mathfrak B$ which assigns to each object $C$ of $\mathscr{C}$ and   collection $\mathfrak B(C) $ of sieves on $\mathscr{C} $, in such a way that 
\begin{enumerate}
\item[($B_1$)] The intersection of two sieves of $\mathfrak B(C)$ contains a sieve of $\mathfrak B(C)$;
\item[($B_2$)] if $\mathcal S$ is a sieve on $\mathfrak B$ and $h: D\rightarrow C$ is any arrow to $C$, then  $h^{*}(\mathcal S) = \{g \mid cod(g) = D,\,\ h\circ g \in \mathcal S\}$  is a sieve on $\mathfrak B(D)$;
\item[$(B_3)$] $\mathfrak B(C)$ is not empty, and the empty sieve is not in $\mathfrak B(C)$.
\end{enumerate}
\end{defi}
\begin{prop}
If $\mathfrak B$ is a basis of filter on a category $\mathscr{C} $, then  $\mathfrak B$ generates a filter $\mathfrak F$ by
\[
S \in \mathfrak F(C) \Leftrightarrow \exists R\in \mathfrak B(C)\,\ \text{such that}\,\ R\subseteq S
\]
for each object $C$ of $\mathscr{C} $.
\end{prop}

It is easy to check that this, indeed, defines a filter from a basis $\mathfrak B$.

\subsection{The ordered set of all filters on a category}
\begin{defi}
Given two filters $\mathfrak F_1$,\,\ $\mathfrak F_2$ on the same category $\mathscr{C}$,  $\mathfrak F_2$ is said to be finer than $\mathfrak F_1$, or $\mathfrak F_1$ is coarser than  $\mathfrak F_2$, if  $\mathfrak F_1(C) \subseteq  \mathfrak F_2(C)$ for all  $C$  object of  $\mathscr{C}$.
\end{defi}
In this way, the set of all filters on a category $\mathscr{C}$ is ordered by the relation {\bf{ ``$\mathfrak F_1$ is coarser than  $\mathfrak F_2$"}}.

Let $(\mathfrak F_i)_{i\in I}$ be a nonempty family of filters on a category $\mathscr{C}$; then the function $\mathfrak F$ which assigns to each object $C$ the collection\linebreak $\mathfrak F(C)=\bigcap_{i\in I}\mathfrak F_i(C)$ is manifestly a filter on  $\mathscr{C}$ and is obviously the greatest lower bound of the family $(\mathfrak F_i)_{i\in I}$ on the ordered set of all filters on   $\mathscr{C}$.
\begin{defi}
An ultrafilter on  a category $\mathscr{C}$ is a filter  such that there is no filter on $\mathscr{C}$ which is strictly finer than $\mathfrak U$.
\end{defi}
Using the Zorn lemma, we deduce that
\begin{prop}\label{finer}
If $\mathfrak F$ is any filter on a category $\mathscr{C}$,  there is an ultrafilter finer than $\mathfrak F$ on $\mathscr{C}$.
\end{prop}
\begin{prop}
Let $\mathfrak{U}$ be an ultrafilter on a category $\mathscr{C}$, and let $C$ be an object of $\mathscr{C}$. Let $ S,T$ be sieves on $C$ such that $S\cup T \in \mathfrak{U}(C)$ then either $S \in \mathfrak{U}(C)$ or $T \in \mathfrak{U}(C)$.
\end{prop}
\begin{proof}
If the affirmation is false, there exist sieves $ S,T$ on $C$ that do not belong to $ \mathfrak{U}(C)$, but $S\cup T \in \mathfrak{U}(C)$.
Consider a function  $\mathfrak{T}:\mathscr{C}\rightarrow Sets$  defined by
$
\mathfrak{T}(C)=\{ R\in \,\ Sieve(C)\mid R\cup S \in \mathfrak{U}(C)\}.
$
Let us verify that $\mathfrak{T}$ is a filter on a  $\mathscr{C}$: in fact, for any object  $C$  of the category $\mathscr{C}$, we have 
\begin{enumerate}
\item [($F_1$)] if $R^{'} \in \mathfrak T(C)$ then $R^{'} \cup S\in \mathfrak{U}(C)$; and if  $ R^{''}$ is a sieve on $C$ such that $ R{'}\subseteq R{''}$, then  $ R^{''}\cup S \in \mathfrak U(C)$. Consequently $R^{''} \in \mathfrak T(C)$.
\item [($F_2$)] We must show that every finite intersection of sieves of $\mathfrak T(C)$ belongs to $\mathfrak T(C)$; indeed,
 let $(R_i)_{i=1,\cdots,n} $
 be a finite collection of sieves on $C$ such that $ R_i\cup S\in \mathfrak{U}(C)$ for all $  i = 1. . . n$, then
$
(R_1\cup S)\cap (R_2\cup S)\cap \cdots \cap (R_n\cup S) = \left(\bigcap_{i=1}^{n}R_i\right)\cup S\in \mathfrak{U}(C).
$
which is equivalent to saying that 
$
\left(\bigcap_{i=1}^{n}R_i\right) \in \mathfrak T(C).
$
\item[($F_3$)] If $R^{'} \in \mathfrak T(C)$ then $R^{'} \cup S\in \mathfrak{U}(C)$; and $h^{*}(R^{'} \cup S)\in \mathfrak T(D)$ for any arrow  $h: D\rightarrow C$; in other words, $h^{*}(R^{'}) \cup h^{*}( S)\in \mathfrak T(D)$, therefore $h^{*}(R^{'})\in \mathfrak T(D)$.
\item [($F_4$)] Evidently, the empty sieve is not in $\mathfrak T(C)$.
 \end{enumerate}
Therefore $\mathfrak{T}$ is a filter finer than $\mathfrak{U}$, since $T\in \mathfrak T(C)$; but this contradicts the hypothesis than 
$\mathfrak{U}$ is an ultrafilter.
\end{proof}
\begin{coro}
If the union  of a finite sequence $(S_i)_{i=1,\cdots,n}$ of sieves on $C$ belongs to the image $\mathfrak{U}(C)$ under  an ultrafilter $\mathfrak{U}$, then at least one of the $S_i$ belongs to $\mathfrak{U}(C)$.
\end{coro}
\begin{proof}
The proof is a simple use of induction on $n$.
\end{proof}

\section{Filters and Grothendieck topologies}
Our main aim in this section is to establish some  connections between filters and Grothendieck topologies in the same category.

First we need some observations.

\begin{defi}\ 
\begin{enumerate}
\item If $J_1$ and $J_2$ are Grothendieck topologies on a category $\mathscr{C}$, we say that $J_1 \preceq J_2$ if and only if $J_1(C)\subseteq J_2(C)$ for all objects $C$ of $ \mathscr{C}$.
\item In the same way, if $\mathfrak F_1$ and $\mathfrak F_2$ are filters on a category $\mathscr{C}$, we say that  $\mathfrak F_1 \preceq \mathfrak F_2$ if and only if $\mathfrak F_1(C)\subseteq \mathfrak F_2(C)$ for all objects $C$ of $ \mathscr{C}$.
\end{enumerate}
 \end{defi}
 It is easy to verify that this definition produces two order relations on Grothendieck topologies and filters respectively.
 \vspace{0.3cm}

In this way, we have the following facts:

\begin{lemma}\label{F-T}
 Every filter on a category  is a Grothendieck topology on the same category.
 \end{lemma}
 \begin{proof}\
 \begin{itemize}
 \item  Given a filter $\mathfrak F$ and an object $C$ of  $\mathscr{C}$, suposse that $S$ is a sieve on $C$. Since $S\subseteq t_{C}$, we have that $t_{C} \in \mathfrak F(C)$.
 \item If $S\in \mathfrak F(C)$, then certainly \,\,\ $h^{*}(S)\in \mathfrak F(D)$ for any arrow  $h: D\rightarrow C$.
 \item If $S\in \mathfrak F(C)$ and $R$ is any sieve on $C$ such that $h^{*}(S)\in \mathfrak F(D)$ for any arrow  $h: D\rightarrow C$ in $S$, then $ h^{*}(S\cap R)=h^{*}(S)\cap h^{*}(R)\in \mathfrak F(D),$ consequently $S\cap R \in \mathfrak F(C)$, and since $S\cap R\subseteq R$ we have $R\in \mathfrak F(C)$.
 \end{itemize}
  \end{proof}
  \begin{lemma}
  Let $\mathfrak F$ be a filter on a category $\mathscr{C}$ and let $J$ be a Gro\-then\-dieck topology on the same category. If $J\preceq \mathfrak F$ then $J$ is a filter.
  \end{lemma}
  \begin{proof}\
  \begin{enumerate}
\item [($F_1$)] Given any object $C$ of $\mathscr{C}$, it is clear that if $S$ is a sieve in $J(C)$ and $R$ is a sieve on $C$ such that $ S \subseteq R$, then  $ R \in J(C)$;
\item [($F_2$)] let  $ (R_i)_{i=1,\cdots,n}$  be a finite collection of sieves on $J(C)$, then $\bigcap_{i=1}^{n} R_i \in J(C)$, and therefore $\bigcap_{i=1}^{n} R_i \in \mathfrak F(C)$ (and consequently is not empty);
\item [($F_3$)] If $S\in J(C)$, then certainly \,\,\ $h^{*}(S)\in J(D)$ for any arrow  $h: D\rightarrow C$.

\item[($F_4$)] the empty sieve is neither in $\mathfrak F(C)$ nor in $J(C)$.
\end{enumerate}
  \end{proof}

\subsection{Product of filters}
Finally in this section we shall be interested in studying a category $\mathscr C$ equipped with a family $( \mathfrak F_{i} )_{i\in I}$ of filters.
\begin{prop}\label{fil-prod}
Let $\mathscr C$ be a category equipped with a family $( \mathfrak F_{i} )_{i\in I}$ of filters, and let $(C_{i})_{i\in I}$ be a family of objects in $\mathscr C$.
Then the function $\mathfrak B$ which assigns to  each object $C=\displaystyle \prod_{i\in I} C_{i}$, the collection of sieves $
\mathfrak B(C)=\Big\{\displaystyle \prod_{i\in I}S_i \mid S_i\in \mathfrak F_{i}(C)\Big\}$, where $S_i=t_{C_i}$ is the maximal sieve on $C_i$  except for a finite number of indices, is basis of a filter on   $\mathscr{C}$.
\end{prop}
\begin{proof}\
\begin{enumerate}
\item[($B_1$)] The formula $\Big( \displaystyle \prod_{i\in I}S_i \Big)\cap \Big(\displaystyle \prod_{i\in I}T_i \Big)= \displaystyle \prod_{i\in I}\left(S_i\cap T_i \right)$ ensure that the intersection of two sieves on $\mathfrak B(C)$ contains a sieve of $\mathfrak B(C)$;
\item[($B_2$)] if $C=\displaystyle \prod_{i\in I} C_{i}$ and $D=\displaystyle \prod_{i\in I} D_{i}$ are objects of $\mathscr C$ and $h: D\rightarrow C$ is any arrow to $C$, since limits commute with limits, we have $h^{*}\Big(\displaystyle \prod_{i\in I}S_i \Big)= \displaystyle \prod_{i\in I}S_i\big(h^{*}(S_i)  \big)$. Therefore  $h^{*}\Big(\displaystyle \prod_{i\in I}S_i \Big)$ is a sieve on $\mathfrak B(D)$.
\item[$(B_3)$] This last assertion is immediate from the first.
\end{enumerate}
\end{proof}
\begin{coro}
The filter of base $\mathfrak B$ \,\ which assigns to  each object $C=\displaystyle \prod_{i\in I} C_{i}$, the collection of sieves $
\mathfrak B(C)=\Big\{\displaystyle \prod_{i\in I}S_i \mid S_i\in \mathfrak F_{i}(C)\Big\}$, where $S_i=t_{C_i}$ is the maximal sieve on $C_i$  except for a finite number of indices, is basis of a filter on   $\mathscr{C}$,  is also generated by the sets ${pr_i}^{-1}(S_i)$, where $S_i$ is a sieve on $C_i$  and $i$ runs through $I$.
\end{coro}
\begin{proof}
It is a consequence of the fact that $ {pr_i}^{-1}(S_i)=S_i\times \displaystyle \prod_{j\neq i} t_{C_j}.$
\end{proof}
\section{Systems of Neighborhoods}

Recall  that a {\bf point} of  an object $C$ of a category $\mathscr{C}$ is a morphism $p:1\rightarrow C$, where $1$ is a terminal object of $\mathscr{C}$.

\begin{defi}
Let  $(\mathscr C, J)$ be a category equipped with a Grothendieck topology, and let $C$  be an object  of $\mathscr{C}$.
A sieve $V$ in $J(C)$,  is said to be a  $\mathfrak{G}$-neighborhood of a point $p:1\rightarrow C$ if there exist a morphism  $\phi:D\rightarrow C$ in  $V$ and  a point $q:1 \rightarrow D $ such that $\phi\circ q =p$.
\end{defi}

\begin{defi}
Let  $(\mathscr C, J)$ be a category equipped with a Grothendieck topology. A cover-neighborhood of $(\mathscr C, J)$ is a function $\mathcal N$ which assigns to each object $(C, J(C))$ of $(\mathscr{C}, J)$ and to each point $p_{\scriptscriptstyle C}:1\rightarrow C$, a collection
$
\mathcal N_{\scriptstyle p_{_{\scriptscriptstyle C}}}(C)$ \ of sieves of\ $\mathscr C
$
such that each sieve in
 $\mathcal N_{\scriptstyle p_{_{\scriptscriptstyle C}}}(C)$ contains a $\mathfrak{G}$-neighborhood of $p_{\scriptscriptstyle C}$.
\end{defi}

\begin{prop}
Let $\mathscr{C}$ be  a  category, and let $C$ be an object of $\mathscr{C}$. The pair $(C,\mathscr N_p(C))$, where  $\mathscr N_p(C)$ is the collection of all  cover-neighbor\-hoods of a point $p:1\rightarrow C$,  is  {\bf a filtered object}
\end{prop}
\begin{proof}\
\begin{enumerate}

\item [(i)] If $S\in\mathscr N_p(C)$ and $R$ is a sieve on $C$ such that $S \subseteq R$, then  $R\in\mathscr N_p(C)$, because there is a $\mathfrak{G}$-neighborhood $V$ of $p_{\scriptscriptstyle C}$ such that $V \subseteq S\subseteq R$;
\item[(ii)] let $\{S_1, S_2,\cdots,S_n\}$ be a finite col\-lec\-tion of sie\-ves of $\mathscr N_p(C)$, then there exists a col\-lec\-tion    $\{V_1, V_2,\cdots, V_n\}$  of $\mathfrak{G}$-neigh\-bor\-hood of $p_{\scriptscriptstyle C}$ such that $V_i\subseteq S_i$ for $I=1,2,\cdots n$, therefore
$\bigcap_{i=1}^n V_n \subseteq\bigcap_{i=1}^n S_n \,\ \text{and}\,\ \bigcap_{i=1}^n S_n\in\mathscr N_p(C)$;
\item[(iii)] the empty sieve is not in $\mathscr N_p(C)$ (each  sieve contains a point $p: 1\rightarrow C$).
\end{enumerate}
\end{proof}
In this case, we say that the point $p: 1\rightarrow C$ is a {\bf limit point} of $\mathscr N_p(C)$. 

\begin{defi}\label{converges}
Let  $(\mathscr C, J)$ be a  category equipped with a Grothendieck topology; let  $\mathfrak F$ be a filter on  $\mathscr{C}$ and let $C$  be an object  of $\mathscr{C}$. 
\begin{enumerate}
\item We shall say that $\mathfrak F(C) $ {\bf converges} to a point $p:1\rightarrow C$ if\linebreak $\mathscr N_p(C)\subseteq \mathfrak F(C) $.
\item The closure of a sieve $A$ on $C$ is the collection of all points $p: 1\rightarrow C$  such that every cover-neighborhood of $p$ meets $A$.
\item A point $p:1\rightarrow C$ is a cluster point of $\mathfrak B(C)$ -the image under the filter base $\mathfrak B$ of $C$- if it lies in the closure of all the sieves on $\mathfrak B(C)$.
\item A point $p:1\rightarrow C$ is a cluster point of $\mathfrak F(C)$ -the image under the filter  $\mathfrak F$ of $C$- if it lies in the closure  of all the sieves on $\mathfrak F(C)$.
\item When  $\mathfrak F(C) $  converges to a point $p:1\rightarrow C$, we shall say that $p: 1\rightarrow C$ is a {\bf limit} point of $\mathfrak F(C)$.
\end{enumerate}
\end{defi}

\begin{defi}
Let  $(\mathscr C, J)$ be a  category equipped with a Grothendieck topology; let  $\mathfrak B$ a basis of a filter on $\mathscr{C}$ and let $C$  be an object  of $\mathscr{C}$. 
The point $p: 1\rightarrow C$ is said to be a limit of $\mathfrak B(C)$ if the image of $C$, by the filter whose base is  $\mathfrak B$, converges to $p: 1\rightarrow C$.
\end{defi}
\begin{prop}
Let  $(\mathscr C, J)$ be a  category equipped with a Grothen\-dieck topology; let  $\mathfrak F$ be a filter on  $\mathscr{C}$ and let $C$  be an object  of $\mathscr{C}$. The point $p: 1\rightarrow C$ is a cluster point of  $\mathfrak F(C)$ if and only if there exists a filter  $\mathscr G$ finer than  $\mathfrak F$ such that $\mathscr G(C) $ {\bf converges} to $p:1\rightarrow C$.
\end{prop}
\begin{proof}
Let us begin by assuming that the point $p:1\rightarrow C$ is a cluster point of  $\mathfrak F(C)$; from definition \ref{converges}, it follows that for each sieve $A$ in $\mathfrak F(C)$, every $\mathfrak{G}$-neighborhood $V$ of $p$ meets $A$. We need to show that the collection  $ \mathscr B(C) =\{A\cap V\mid V \,\ \text{is a $\mathfrak{G}$-neighborhood of}\,\ p\}$ define  a base for a filter  $\mathscr G$ finer than $\mathfrak F$.in such a way that  $\mathscr G(C) $ {\bf converges} to $p: 1\rightarrow C$.

Indeed,
\begin{enumerate}
\item[($B_1)$] Let $A\cap V$,\,\ $A\cup W$ two elements of the collection $\mathscr B(C)$, since $(A\cap V) \cap (A\cup W)= A\cup (V\cap W) $ and $V\cap W$ is a $\mathfrak{G}$-neighborhood of $p$, there exists $U$,  a $\mathfrak{G}$-neighborhood of $p$ such that $U\subseteq V\cap W,$ and clearly $A\cap U \in \mathscr B(C)$;
\item[$(B_2)$] Obviously $\mathscr B(C)$ is not empty, and the empty sieve is not in $\mathscr B(C)$.
\end{enumerate}
Now, if $\mathscr G$ is the filter generated by $\mathscr B$ then  $\mathscr G$ is finer than  $\mathfrak F$, and $\mathscr G(C)$ naturally converges to $p: 1\rightarrow C$.

Conversely, if there is a filter  $\mathscr G$ finer than  $\mathfrak F$ such that $\mathscr G(C) $ {\bf converges} to $p: 1\rightarrow C$ then each sieve $R$ in $\mathfrak F(C)$ and each $\mathfrak{G}$-neighbor\-hood $U$ of $p: 1\rightarrow C$ belongs to  $\mathscr G$ and hence meet, so the point $p: 1\rightarrow C$ is a cluster point of  $\mathfrak F(C)$.
\end{proof}
\begin{prop}
Let  $(\mathscr C, J)$ be a  category equipped with a Grothen\-dieck topology; let  $C$  be an object  of $\mathscr{C}$ and let $A$ be a sieve on  $C$. The point $p: 1\rightarrow C$ lies in the closure  of $A$ if and only if there is a filter  $\mathscr G $ such that  $A\in\mathscr G(C) $ and $\mathscr G(C) $ {\bf converges} to $p: 1\rightarrow C$.
\end{prop}
\begin{proof}
Let us begin by assuming that the point $p:1\rightarrow C$ lies in the closure  of $A$; from definition \ref{converges}, it follows that  every $\mathfrak{G}$-neighborhood $V$ of $p$ meets $A$. Then
$ \mathscr B(C) =\{ A\cap V\mid V \,\ \text{is a $\mathfrak{G}$-neighborhood of $p$} \}
$ 
is a base  for a filter  $\mathscr G$, in such a way that  $\mathscr G(C) $ {\bf converges} to $p:1\rightarrow C$.

Conversely, if  $A \in \mathscr G(C)$ and  $\mathscr G(C) $ {\bf converges} to $p:1\rightarrow C$ then  $p:1\rightarrow C$ is a cluster point  of $\mathscr G(C) $ and hence $p:1\rightarrow C$ lies in the closure  of $A$.
\end{proof}
\begin{coro}
Let $\mathfrak{U}$ be an ultrafilter on a category $\mathscr{C}$, and let $C$ be an object of $\mathscr{C}$.  $\mathfrak U(C)$ converges to a point $p:1\rightarrow C$ if and only if $p:1\rightarrow C$ is a cluster point of  $\mathfrak U(C)$.
\end{coro}
\begin{ex}
Let $\mathcal A$ be a complete Heyting algebra and regard  $\mathcal A$  as a category in the usual way. 
\begin{itemize}
\item  Then (see \cite{MM}) $\mathcal A$ can be equipped with a base for a Grothendieck topology $K$, given by
$\{a_i\mid i\in I\}\in K(c)$ if and only if $\bigvee_{i\in I}=c$,
where $\{a_i\mid i\in I\}\subseteq \mathcal A$ and $c\in \mathcal A$.

\item A sieve $S$ on an element $c$ of  $\mathcal A$ is just a subset of elements $b\leqslant c$ such that $a\leqslant b\in S$ implies $a\in S$.
\item In the Grothendieck topology $J$ with basis $K$, a sieve $S$ on $c$ covers $c$  iff
$
\bigvee S=c.
$
\item A filter on $\mathcal A$ is a function $\mathfrak F$ which assigns to each element $c$ of $\mathcal A$ a  collection $\mathfrak F(c) $ of sieves, such that 
\begin{enumerate}
\item [($F_1$)] If $S \in \mathfrak F(c)$ and $ R$ is a sieve on $c$ such that $ S \subseteq R$, then  $ R \in \mathfrak F(c)$;
\item [($F_2$)] every finite intersection of sieves of $\mathfrak F(c)$ belongs to $\mathfrak F(c)$;
\item [($F_3$)] the empty sieve is not in $\mathfrak F(c)$.
\end{enumerate}
\item An immediate consequence of the previous construction of  a Grothendieck topology and a filter on $\mathcal A$ is that 
$ \mathfrak F(c)$ converges to $c$ iff $\bigvee S=c$, for each $ S \in \mathfrak F(c).
$
\end{itemize}
\end{ex}

\section{Filter-preserving functors}
\begin{defi}
Let $(\mathscr C, \mathfrak F)$ and $(\mathscr D, \mathfrak G)$  be  small categories  equipped with filters and $F:\mathscr C\longrightarrow \mathscr D$ a functor. We say $F$ is filter-preserving (or continuous)  if, for any $c\in ob(\mathscr C)$ and any covering sieve \linebreak $R\in \mathfrak{F}(c)$, the family $\{ F(f)\mid f\in R\}$ generates a covering sieve $S \in \mathfrak{G}\left(F(c)\right)$, consisting of all the morphisms with codomain $F(c)$ which factor through at least one $F(f)$.
\end{defi}
We shall use the notation $\langle F(R)\rangle$     to denote the covering sieve generated by the family $\{ F(f)\mid f\in R\}$ in $ \mathfrak{G}\left(F(c)\right)$.
\begin{prop}\label{ne-im}
Let $F: (\mathscr C, J)\longrightarrow (\mathscr D, K)$ be a morphism of sites. If, for every  object $C$ of $\mathscr C$,\,\  $V$ is a $\mathfrak G$-neighborhood of a point $p: 1\rightarrow C$, then the family $F(V)=\{ F(\alpha)\mid \alpha \in V\}$ generates  a $\mathfrak G$-neighborhood $\langle F(V)\rangle$ of the point $F(p)$  of $F(C)$ in $\mathscr D$.
\end{prop}

\begin{proof}
The hypothesis that $V$ is a $\mathfrak G$-neighborhood of a point $p: 1\rightarrow C$  tell us that there exists a morphism  $\phi:D\rightarrow C$ in  $V$ and  a point $q:1 \rightarrow D $ such that $\phi\circ q=p$.

Next, we apply functor $F$ to obtain $F(\phi)\circ F(q)= F(p)$, 
where , of course, $F(\phi)$ is in                                                                                                                                                          the $\mathfrak G$-neighborhood $\langle F(V)\rangle$ of the point $F(p)$. 
\end{proof}
 \begin{prop}\label{cov-ne-im}
 Let $F: (\mathscr C, J)\longrightarrow (\mathscr D, K)$ be a morphism of sites and let  $\mathscr N$ be a cover-neighborhood of $(\mathscr C, J)$ then $\langle  F(\mathscr N)  \rangle$ is a a cover-neighborhood of $(\mathscr D, K)$.
\end{prop}
 \begin{proof}
  Let $V$ in $J(C)$,  be  a  $\mathfrak{G}$-neighborhood  of a point $p_{\scriptscriptstyle C}: \rightarrow C$, and let $W$ a sieve in $ \mathscr N_{\scriptstyle p_{_{\scriptscriptstyle C}}}(C)$, where  $\mathscr N$ is a cover-neighborhood of $(\mathscr C, J)$, such that $V \hookrightarrow W$ is an inclusion.
Since $F$ is a morphisms of sites, $F(V) \hookrightarrow F(W)$ is also an inclusion which belongs to $\langle  \mathscr N_{\scriptstyle p_{_{\scriptscriptstyle C}}}(C)  \rangle$, and clearly $\langle  F(\mathscr N)  \rangle$ is a a cover-neighborhood of $(\mathscr D, K)$.
  
 \end{proof}
\begin{prop}\label{fil-im}
 Let $F: (\mathscr C, J)\longrightarrow (\mathscr D, K)$ be a morphism of sites and let  $\mathfrak B$ be a  basis of a filter on the category $\mathscr{C}$ then $ F(\mathfrak B (C))$ is a basis of  filter on the category $\mathscr{D}$.
\end{prop}
\begin{proof}
let $C$  be an object  of $\mathscr{C}$ 
We shall show that 
\begin{enumerate}
\item[($B_1$)] The intersection of two sieves on $F(\mathfrak B (C))$ contains a sieve on $F(\mathfrak B (C))$;
\item[($B_2$)] If $S$ is a sieve on $F(\mathfrak B (C))$ and $h: F(D)\rightarrow F(C)$ is any arrow to $F(C)$, then  $\big(F(h)\big)^{*} \big( F(S)\big)$ is a sieve on $F\big(\mathfrak B(D)\big)$.
\item[$(B_3)$] $F(\mathfrak B(C))$ is not empty, and the empty sieve is not on $F(\mathfrak B(C))$.
\end{enumerate}
First, since a morphisms of sites preserves inclusion, for sieves  $S, T$ on $\mathfrak B (C)$, we have $F(S\cap T) \subseteq F(S)\cap F(T)$.

Next, let $S$ be a sieve on $\mathfrak B (C)$ and let $h: D\rightarrow C$ be any arrow to $C$ such that  $h^{*}( S) $  is a sieve on $\mathfrak B(D))$, then $F(S)$ is a sieve on $F\big(\mathfrak B (C)\big)$ and for $F(h): F(D)\rightarrow F(C)$, we have that ${F(h)}^{*}\big( F(S)\big)$ is a sieve on $F\big(\mathfrak B(D)\big)$.

Finally, for every sieve $S$  on $\mathfrak B (C)$, the fact that $S\neq\emptyset$ implies $F(S)\neq\emptyset$.
\end{proof}
\section{Compactness on small categories}
\begin{defi}\label{compact}
Let  $(\mathscr C, J)$ be a  site and let $C$  be an object  of $\mathscr{C}$. We say that an object $C$    of $\mathscr{C}$
 \begin{itemize}
 \item  Is {\bf  quasi-compact} if, for every filter $\mathfrak F$  on  $\mathscr{C}$,  $\mathfrak F(C)$ has at least one cluster point.
 \item Is  {\bf Hausdorff} if, for every filter $\mathfrak F$  on  $\mathscr{C}$,  $\mathfrak F(C)$ has no more that one limit point. 
 \item Is  {\bf compact} if, for every filter $\mathfrak F$  on  $\mathscr{C}$,  $\mathfrak F(C)$ is quasi-compact and Hausdorff.
 \end{itemize}
\end{defi}
\begin{lemma}\label{u-comp}
Let  $(\mathscr C, J)$ be a  site and let $C$  be an object  of $\mathscr{C}$. An object $C$    of $\mathscr{C}$   is compact if and only if, for every ultrafilter $\mathfrak U$  on  $\mathscr{C}$,  $\mathfrak U(C)$ is convergent.
\end{lemma}
\begin{proof}
First suppose that $\mathfrak F$  is a filter on  $\mathscr{C}$. Proposition \ref{finer}
ensures that, for  every filter, there exists an ultrafilter $\mathfrak U$  finer than $\mathfrak F$, such that $\mathfrak U(C)$ converges to a point $p$ on $C$, therefore $p$ is a cluster point of $\mathfrak F(C)$.

Conversely, if, for an ultrafilter $\mathfrak U$, $\mathfrak U(C)$ has a cluster point then it converges to this point.
\end{proof}
\begin{prop}
Let $(\mathscr C, \mathfrak F)$ and $(\mathscr D, \mathfrak G)$  be  small categories  equipped with filters and $F:\mathscr C\longrightarrow \mathscr D$ a filter-preserving (continuous) functor. If  $C$ is a compact object   of $\mathscr{C}$, then   $\mathfrak F(C)$ is a compact object of $\mathscr D$.
\end{prop}
\begin{proof}
This is a consequence of propositions \ref{ne-im}, \ref{cov-ne-im}, \ref{fil-im} and lemma \ref{fil-im}.
\end{proof}
\begin{teor}(Tychonoff)
 Let  $(\mathscr C, J)$ be a  site. Every product of compact objects in the  category $\mathscr C$ is compact. 
\end{teor}
\begin{proof}
Suppose we have  chosen a collection $(C_{i})_{i\in I}$  of compact objects in $\mathscr C$; equivalently, for every family $( \mathfrak U_{i} )_{i\in I}$ of ultrafilters  on   $\mathscr{C}$, $\mathfrak U_i(C_i)$ is convergent.
Then the function $\mathfrak B$ which assigns to  each object\linebreak $C=\displaystyle \prod_{i\in I} C_{i}$, the collection of sieves $ \mathfrak B(C)=\Big\{\displaystyle \prod_{i\in I}S_i \mid S_i\in \mathfrak F_{i}(C)\Big\}$, where $S_i=t_{C_i}$ is the maximal sieve on $C_i$  except for a finite number of indices, is basis of an ultrafilter on   $\mathscr{C}$.
\end{proof}

\section{Filters and compactness on locales}

There is a significant `generalization' of the notion of topological space,
namely the notion of locale.
The object of this section is to present an alternative approach to the notion of
compactness on locales, different from (but not entirely independent of) the one which we have followed previously.

From  P.J. Johnstone ({\cite{PJ1}) and A. J. Lindenhovius  (\cite{AL}), we take the following ideas
\begin{defi}\
\begin{enumerate}
\item  The category  ${\bf Frm}$ of frames is the category whose objects are complete lattices satisfying the infinite distributive law, and \linebreak whose morphisms are functions preserving finite meets and arbi\-trar\-y joins.
\item The category ${\bf Loc}$ of locales is the opposite of the category ${\bf Frm}$. We refer to morphisms in ${\bf Loc}$ as continuous maps, and write $\Omega$ for the functor ${\bf Sp}\rightarrow {\bf Loc}$ which sends a topological  space to its lattice of open sets, and a continuous map $f:X\rightarrow Y$ to the function 
$f^{-1}: \Omega(Y)\rightarrow \Omega(X)$.
\end{enumerate}
\end{defi}

\begin{defi}
Let $(P,\leqslant)$ be a lattice and $M \subseteq P$. We say that
\begin{itemize}
\item  $a\in P$ is a join (or least upper bound) for $M$, and write \linebreak $a=\bigvee M$ if
\begin{enumerate}
\item $a$ is an upper bound for $M$, i.e. $m\leqslant a$ for all $m\in M$, and
\item if $b$ satisfies for all $m\in M \ (m\leqslant b)$ then $a\leqslant b$.
\end{enumerate}
\item  If $M$ is a two-element set $\{m, k\}$, we write $m\vee k$ for $\bigvee\{m,k\}$ \item  if $M$ is the empty set $\emptyset$, we write $0$ for $\bigvee \emptyset$,  clearly $0$ is just the least element of $M$.
\item Dually, in any lattice we can consider the notion of meet  (greatest lower bound), defined by reversing all the inequalities in the
definition of join. We write $\bigwedge M$,\,\ $a\land b$ and $1$ for the analogues of $\bigvee M$,\.\ $a\vee b$ and $0$. 

\item  $M$ is an up-set if for each $x \in M$ and $y \in P$ we have $x \leqslant y$ implies $y \in M$.
\item  Similarly, $M$ is called a down-set if for each $x \in M$ and $y\in P$ we have $y\leqslant x$ implies $y \in M$. 
\item Given an element $x \in P$, we define the up-set and down-set generated by $x$ by $\uparrow x = \{y \in P : x\leqslant y\}$ and $\downarrow x = \{y \in P : y \leqslant x\}$, respectively.
\item We can also define the up-set generated by a subset $M$ of $P$ by  $\uparrow M = \{x \in P : m \leqslant x \,\ \text{for some}\,\ m \in M\} =\bigcup_{m\in M} \uparrow m$, and similarly, we define the down-set generated by $M$ by \linebreak $\downarrow M =\bigcup_{m\in M} \downarrow m$.
\end{itemize}
\end{defi}
\begin{notation}
We denote the collection of all up-sets of a partially ordered set $P$ by $\mathcal U(P)$ and the set of all down-sets by $\mathcal D(P)$.
\end{notation} 

\begin{ex}
For $Z^{+}_{D} :=(\mathbb Z^{+}, \leqslant_{|})$, where $\leqslant_{|}$ is the  multiplicative \linebreak (or divisibility) partial order on  the set $\mathbb Z^{+}$ of positive integers, we can observe that
\begin{itemize}
\item The down-set generated by an element  $n$ of $\mathbb Z^{+}$ is the collection $\downarrow n = \{k \in \mathbb Z^{+} : k \leqslant_{|} n\}= \mathfrak D_n$, the set of all divisors of $n$.
\item If $M\subseteq  \mathbb Z^{+}$ then the down-set generated by $M$ is\linebreak $\downarrow M =\bigcup_{m\in M} \mathfrak D_m$.
\item The up-set generated by $n\in \mathbb Z^{+}$ is $\uparrow n = \{k \in \mathbb Z^{+} : n \leqslant_{|} k\}$, i.e. the set $\mathfrak M_n$ of all multiples of $n$.
\item If $P\subseteq  \mathbb Z^{+}$ then the up-set generated by $P$ is $\uparrow M =\bigcup_{p\in P} \mathfrak M_p$.
\item An interesting fact about $\mathfrak D_n$, the set of all divisors of $n$,  is that it is a {\bf locale}.
\item Using the Fundamental Theorem of Arithmetic it is easy to show  that  if $n\in \mathbb Z^{+}$ is the product of a finite number of distinct primes then $\mathfrak D_n$ is a Boolean Algebra.
\end{itemize}
\end{ex}

\subsection{Grothendieck topologies on  locales}
In order to construct Grothendieck topologies, we first define sieves.
\begin{defi}
Given an element $k$ in a locale $L$, a subset $S$ of $L$ is called a sieve on $k$ if $S\in \mathcal D(\downarrow k)$.
\end{defi}
\begin{defi}
A Grothendieck topology  on  a locale $L$ is a function $J$ which assigns to each object $k$ of $L$  a collection $J(k)$ of sieves on $k$, in such a way that
\begin{enumerate}
\item[(i)] the maximal sieve $\downarrow k$  is in $J(k)$;
\item[(ii)] (stability axiom) if $S\in J(k)$ and $m\leqslant k$ then $S\cap (\downarrow m)$ is in $J(m)$;
\item[(iii)] (transitivity axiom) if $S\in J(k)$ and $R$ is any sieve on $k$ such that $R\cap (\downarrow m)$ is in $J(m)$ for each $m\in S$, then $R\in J(k)$.
\end{enumerate}
\end{defi}
\begin{ex}\
\begin{itemize}
\item The trivial Grothendieck topology on $L$ is given by $J_{tri}(n) = \downarrow n$.
\item The discrete Grothendieck topology on the lattice $L$ is given by\linebreak $J_{dis}(n) = \mathcal D(\downarrow n) $.
\item The atomic Grothendieck topology on $L$ can only be defined if $L$ is downwards
directed, and is given by $J_{atom}(n) = \mathcal D(\downarrow n)- \{\emptyset\}$.
\item A subset $D\subseteq (\downarrow n)$ is said to be dense below $n$ if for any $m\leqslant n$ there exists $k\leqslant m$ with $k\in D$. 
\item The dense topology on $L$ is given by $$J(n)= \{ D\mid k\leqslant n\,\ \text{for all}\,\ k\in D,\,\ \text{and $D$ is a sieve dense below}\,\ n\}.$$
\end{itemize}
\end{ex}

\subsection{$S$-Filters on  locales}

Recall that a basic notion in complete lattices is that of {\bf filter}:
a filter $F$ of $L$ is a non-empty subset of $L$ such that
\begin{enumerate}
\item $F$ is a sublattice of $ L$ ,  and
\item for any $a\in F$ and $b\in L$, $a\lor b \in  F$.
\end{enumerate}
In a different direction, we have the notion of {\bf filter of sieves (ideals)} on a locale $L$, which  we shall call  {\bf $S$-filter}.
\begin{defi}
An $S$-filter   on  a locale $L$ is a function $\mathfrak F$ which assigns to each object $k$ of $L$  a collection $\mathfrak F(k)$ of sieves on $k$, in such a way that
\begin{enumerate}
\item[($F_{1}$)] if $S\in \mathfrak F(k)$ and $R$ is a sieve on $k$ such that $S\subseteq R$, then   $R\in \mathfrak F(k)$;
\item[($F_{2}$)] every finite intersection of sieves of $\mathfrak F(k)$ belongs to $\mathfrak F(k)$;
\item[($F_3$)] if $S\in J(k)$ and $m\leqslant k$ then $S\cap (\downarrow m)$ is in $\mathfrak F(m)$;
\item[($F_4$)] the empty sieve is not in $\mathfrak F(k)$.
\end{enumerate}
\end{defi}
\begin{ex}
From the definition of a Grothendieck topology $J$ on a locale $L$  it follows that for each object $k$ of $L$ and that
\begin{itemize}
\item for $S\in J(k)$ any larger sieve $R$ on $C$ is also a member of $J(k)$. 
\item It is also clear that if $R; S\in J(k)$ then $R\cap S \in J(k)$, 
\item consequently some Grothendieck topologies produce  $S$-filters in the same locale $L$: they are exactly those for which   
\be \label{non-dis}
R\cap S\neq\emptyset\,\,\ \text{for all pairs}\,\,\ R; S\in J(k)
\ee
 and such that  the empty sieve is not in $J(k)$.
\item Clearly, the {\bf trivial topology} on $k$ is a $S$-filter we shall call it {\bf trivial $S$-filter}.
\item It is also clear that the {\bf atomic topology} on $L$ (see \cite{MM})  is not a $S$-filter.
\end{itemize}
\end{ex}
\begin{defi}
A basis of an $S$-filter  on a locale $L$ is a function $\mathfrak B$ which assigns to each  $k\in L$ a  collection $\mathfrak B(k) $ of sieves on $k $, in such a way that 
\begin{enumerate}
\item[($B_1$)] The intersection of two sieves of $\mathfrak B(k)$ contains a sieve of $\mathfrak B(k)$;

\item[($B_2$)] if $S$ is a sieve on $\mathfrak B(k)$ and $m\leqslant k$ then $S\cap (\downarrow m)$ is in $\mathfrak B(m)$;

\item[$(B_3)$] $\mathfrak B(k)$ is not empty and the empty sieve is not in $\mathfrak B(k)$.
\end{enumerate}
\end{defi}
\begin{prop}
If $\mathfrak B$ is a basis of $S$-filter  on a locale $L$, then  $\mathfrak B$ generates an $S$-filter  $\mathfrak F$ by
$
S \in \mathfrak F(k) \Leftrightarrow \exists R\in \mathfrak B(k)\,\ \text{such that}\,\ R\subseteq S
$
for each object $k\in L$.
\end{prop}

It is easy to check that this, indeed, defines a $S$-filter  from a basis $\mathfrak B$.

\subsection{The ordered set of all $S$-filters on a locale}
\begin{defi}
Given two $S$-filters $\mathfrak F_1$,\,\ $\mathfrak F_2$ on the same locale $L$,\  $\mathfrak F_2$ is said to be finer than $\mathfrak F_1$, or $\mathfrak F_1$ is coarser than  $\mathfrak F_2$, if  $\mathfrak F_1(k) \subseteq  \mathfrak F_2(k)$ for all  $k \in L$.
\end{defi}
In this way, the set of all $S$-filters on a locale $L$ is ordered by the relation {\bf{ ``$\mathfrak F_1$ is coarser than  $\mathfrak F_2$"}}.

Let $(\mathfrak F_i)_{i\in I}$ be a nonempty family of $S$-filters on a $L$; then the function $\mathfrak F$ which assigns to each object $k\in L$ the collection $\mathfrak F(k)=\bigcap_{i\in I}\mathfrak F_i(k)$ is manifestly a $S$-filter on  $L$ and is obviously the greatest lower bound of the family $(\mathfrak F_i)_{i\in I}$ on the ordered set of all $S$-filters on   $L$.
\begin{defi}
An $S$-ultrafilter on  a locale $L$ is a $S$-filter $\mathfrak U$ such that there is no $S$-filter on $L$ which is strictly finer than $\mathfrak U$.
\end{defi}
Using the Zorn lemma, we deduce that
\begin{prop}\label{L-finer}
If $\mathfrak F$ is any $S$-filter on a locale $L$,  there is an $S$-ultrafilter finer than $\mathfrak F$ on $L$.
\end{prop}
\begin{prop}
Let $\mathfrak{U}$ be an $S$-ultrafilter on a locale $L$, and let $k\in L$ . Let $ S,T$ be sieves on $k$ such that $S\cup T \in \mathfrak{U}(k)$ then either $S \in \mathfrak{U}(k)$ or $T \in \mathfrak{U}(k)$.
\end{prop}

\begin{proof}
If the affirmation is false, there exist sieves $ S,T$ on $k$ that do not belong to $ \mathfrak{U}(k)$, but $S\cup T \in \mathfrak{U}(k)$.
Consider a function  $\mathfrak{T}:L\rightarrow Sets$  defined by
$
\mathfrak{T}(k)=\{ R\in \,\ Sieve(k)\mid R\cup S \in \mathfrak{U}(k)\}.
$

Let us verify that $\mathfrak{T}$ is a $S$-filter on a  $L$: in fact, for any object  $k$  of  $L$, we have 
\begin{enumerate}
\item [($F_1$)] if $R^{'} \in \mathfrak T(k)$ then $R^{'} \cup S\in \mathfrak{U}(k)$; and if  $ R^{''}$ is a sieve on $k$ such that $ R{'}\subseteq R{''}$, then  $ R^{''}\cup S \in \mathfrak U(k)$. Consequently $R^{''} \in \mathfrak T(k)$.
\item [($F_2$)] We must show that every finite intersection of sieves of $\mathfrak T(k)$ belongs to $\mathfrak T(k)$; indeed,
let $ (R_i)_{i=1,\cdots,n}$  be a finite collection of sieves on $k$ such that 
$
R_i\cup S\in \mathfrak{U}(k),\,\ \text{for all}\,\  i = 1. . . n,
$
then
$
(R_1\cup S)\cap (R_2\cup S)\cap \cdots \cap (R_n\cup S) = \left(\bigcap_{i=1}^{n}R_i\right)\cup S\in \mathfrak{U}(k).
$
which is equivalent to saying that 
$
\left(\bigcap_{i=1}^{n}R_i\right) \in \mathfrak T(k).
$
\item[($F_3$)] If $R^{'} \in \mathfrak T(k)$ then $R^{'} \cup S\in \mathfrak{U}(k)$; and $h^{*}(R^{'} \cup S)\in \mathfrak T(n)$ for any arrow  $h: n\rightarrow k$; in other words, $h^{*}(R^{'}) \cup h^{*}( S)\in \mathfrak T(n)$, therefore $h^{*}(R^{'})\in \mathfrak T(n)$.
\item [($F_4$)] Evidently, the empty sieve is not in $\mathfrak T(k)$.
 \end{enumerate}
Therefore $\mathfrak{T}$ is a $S$-filter finer than $\mathfrak{U}$, since $T\in \mathfrak T(k)$; but this contradicts the hypothesis than 
$\mathfrak{U}$ is an $S$-ultrafilter.
\end{proof}
\begin{coro}
If the union  of a finite sequence $(S_i)_{i=1,\cdots,n}$ of sieves on $k$ belongs to the image, $\mathfrak{U}(k)$,  of an object $k$ under  an ultra$S$-filter $\mathfrak{U}$, then at least one of the $S_i$ belongs to $\mathfrak{U}(k)$.
\end{coro}
\begin{proof}
The proof is a simple use of induction on $n$.
\end{proof}

\subsection{$S$-filters and Grothendieck topologies on locales}
\begin{defi}\ 
\begin{enumerate}
\item If $J_1$ and $J_2$ are Grothendieck topologies on a locale $L$, we say that $J_1 \preceq J_2$ if and only if $J_1(k)\subseteq J_2(k)$ for all $k\in L$.
\item In the same way, if $\mathfrak F_1$ and $\mathfrak F_2$ are $S$-filters on  $L$, we say that  $\mathfrak F_1 \preceq \mathfrak F_2$ if and only if $\mathfrak F_1(k)\subseteq \mathfrak F_2(k)$ for all objects $k\in L$.
\end{enumerate}
 \end{defi}
 It is easy to verify that this definition produces two order relations on Grothendieck topologies and $S$-filters respectively.
  
In this way, we have the following facts:
\begin{lemma}\label{F-T-L}
 Every $S$-filter on a locale  is a Grothen\-dieck topology on the same locale.
\end{lemma}
\begin{lemma}
  Let $\mathfrak F$ be a $S$-filter on a locale $L$ and let $J$ be a Grothen\-dieck topology on the same locale. If $J\preceq \mathfrak F$ then $J$ is a $S$-filter.
  \end{lemma}

 \begin{proof}\
  \begin{enumerate}
  \item [($F_1$)] Given any  $k\in L$, it is clear that if $S$ is a sieve in $J(k)$ and $R$ is a sieve on $k$ such that $ S \subseteq R$, then  $ R \in J(k)$;
\item [($F_2$)] let  $ (R_i)_{i=1,\cdots,n}$  be a finite collection of sieves on $J(k)$, then $\bigcap_{i=1}^{n} R_i \in J(k)$, and therefore $\bigcap_{i=1}^{n} R_i \in \mathfrak F(k)$ (and consequently is not empty);
\item [($F_3$)] If $S\in J(k)$, and $m\leqslant k$ then certainly \,\,\ $S\cap (\downarrow m)$ is in $\mathfrak F(m)$;
\item[($F_4$)] the empty sieve is neither in $\mathfrak F(k)$ nor in $J(k)$.
\end{enumerate}
 \end{proof}

 \subsection{Systems of Neighborhoods on Locales}
 
 The basic ideas used in in this section can be found in  P.T.  Jhonstone  book  ``Stone Spaces''\linebreak({\cite{PJ2} pages 41-42) or in the S. MacLane and  I. Moerdijk book ``Sheaves in Geometry and Logic'' (see \cite{MM} pages 470-473):  A point of a topological space $X$ is the same thing as a continuous map $1 \rightarrow X$, where $1$ is the one-point space, it seems reasonable to define a point of a locale $L$ to be a continuous map  $\Omega(1) = 2\rightarrow L$, i.e. a frame homomorphism $p: L\rightarrow 2$.  such a map is completely determined by its kernel
$p^{-1} (0)$ or its dual kernel $p^{-1}(1)$, which are   respectively a prime ideal and a prime filter of $L$. 

In what follows, we are going to use the frame homomorphisms\linebreak
$p: L\rightarrow 2$ as  points of the locale $L$. 
\begin{defi}
Let  $( L, J)$ be a locale equipped with a Grothen\-dieck topology, let $k\in L$ and let $p$ be a point of $L$ such that $k\in p^{-1}(0)$. A sieve $V$ in $J(k)$,  is said to be a  $\mathfrak{G}$-neighborhood  of $p$ if $V\subseteq p^{-1}(0)$.

\end{defi}

\begin{defi}
Let  $(L, J)$ be a locale equipped with a Grothendieck topology. A cover-neighborhood of $(L, J)$ is a function $\mathcal N$ which assigns to each object $(k, J(k))$ of $(L, J)$ and to each point $p$ of $L$, for which $k\in p^{-1}(0)$, a collection 
$
\mathcal N_{\scriptstyle p}(k)\,\ \text{of sieves of}\,\ k
$
such that each sieve in
 $\mathcal N_{\scriptstyle p}(k)$ contains a $\mathfrak{G}$-neighborhood of $p$.
\end{defi}

\begin{prop}\label{neighbor}
Let $L$ be  a  locale, and let $k\in L$. Then the pair $(k,\mathscr N_p(k))$, where  $\mathscr N_p(k)$ is the collection of all  cover-neighbor\-hoods of a point $p$ (for which  $k\in p^{-1}(0)$)  is  an object of $L$ equipped with cover-neighborhood.
\end{prop}
\begin{proof}\
\begin{enumerate}

\item [(i)] If $S\in\mathscr N_p(k)$ and $R$ is a sieve on $k$ such that $S \subseteq R$, then  $R\in\mathscr N_p(k)$, because there is a $\mathfrak{G}$-neighborhood $V$ of $p_{\scriptscriptstyle C}$ such that $V \subseteq S\subseteq R$;
\item[(ii)] let $\{S_1, S_2,\cdots,S_n\}$ be a finite col\-lec\-tion of sie\-ves of $\mathscr N_p(k)$, then there exists a col\-lec\-tion    $\{V_1, V_2,\cdots, V_n\}$  of $\mathfrak{G}$-neigh\-bor\-hood of $p_{\scriptscriptstyle C}$ such that $V_i\subseteq S_i$ for $I=1,2,\cdots n$, therefore\linebreak
$\bigcap_{i=1}^n V_n \subseteq\bigcap_{i=1}^n S_n $  and $\bigcap_{i=1}^n S_n\in\mathscr N_p(Ck)$;
\item[(iii)] the empty sieve is not in $\mathscr N_p(k)$ (each  sieve contains a point).
\end{enumerate}
\end{proof}

In this case, we say that the point $p$ of $l$ is a {\bf limit point} of $\mathscr N_p(k)$. 

\begin{defi}
Let  $(L, J)$ be a  locale equipped  with a Grothendieck topology; let  $\mathfrak F$ be a $S$-filter on  $L$ and let $k\in L$. 
\begin{enumerate}
\item We shall say that $\mathfrak F(k) $ {\bf converges} to a point $p$ of $L$, for which $k\in p^{-1}(0)$, if $\mathscr N_p(k)\subseteq \mathfrak F(k) $.
\item The closure of a sieve $A$ on $k$ is the collection of all points $p$ of $L$ (satisfying  $k\in p^{-1}(0)$) such that every cover-neighborhood of $p$ meets $A$.
\item A point $p$  is a cluster point of $\mathfrak B(k)$ -the image under the $S$-filter base $\mathfrak B$ of $k$- if it lies in the closure of all the sieves on $\mathfrak B(k)$.
\item A point $p$ of $C$ is a cluster point of $\mathfrak F(C)$ -the image under the $S$-filter  $\mathfrak F$ of $C$- if it lies in the closure  of all the sieves on $\mathfrak F(C)$.
\end{enumerate}
\end{defi}

\begin{prop}\label{converges-L}
Let  $(L, J)$ be a  locale equipped with a Grothendieck topology; let  $\mathfrak F$ be a $S$-filter on  $L$ and let $k\in L$. The point $p$ of $L$ is a cluster point of  $\mathfrak F(k)$ if and only if there exists a $S$-filter  $\mathscr G$ finer than  $\mathfrak F$ such that $\mathscr G(k) $ {\bf converges} to $p$.
\end{prop}
\begin{proof}
Let us begin by assuming that the point $p$ of $L$ is a cluster point of  $\mathfrak F(k)$; from definition \ref{converges}, it follows that for each sieve $A$ in $\mathfrak F(k)$, every $\mathfrak{G}$-neighborhood $V$ of $p$ meets $A$. We need to show that the collection 
$
\mathscr B(k) =\{A\cap V\mid V \,\ \text{is a $\mathfrak{G}$-neighborhood of}\,\ p\}
$
define  a base for an $S$-filter  $\mathscr G$ finer than $\mathfrak F$.in such a way that  $\mathscr G(C) $ {\bf converges} to $p$.

Indeed,
\begin{enumerate}
\item[($B_1)$] Let $A\cap V$,\,\ $A\cup W$ two elements of the collection $\mathscr B(k)$, since 
$(A\cap V) \cap (A\cup W)= A\cup (V\cap W) $
and $V\cap W$ is a $\mathfrak{G}$-neighborhood of $p$, there exists $U$,  a $\mathfrak{G}$-neighborhood of $p$ such that
$U\subseteq V\cap W,$ and clearly $A\cap U \in \mathscr B(k)$;
\item[$(B_2)$] Obviously $\mathscr B(k)$ is not empty, and the empty sieve is not in $\mathscr B(C)$.
\end{enumerate}

Now, if $\mathscr G$ is the $S$-filter generated by $\mathscr B$ then  $\mathscr G$ is finer than  $\mathfrak F$, and $\mathscr G(C)$ naturally converges to $p$.

Conversely, if there is a $S$-filter  $\mathscr G$ finer than  $\mathfrak F$ such that $\mathscr G(k) $ {\bf converges} to $p$ then each sieve $R$ in $\mathfrak F(k)$ and each $\mathfrak{G}$-neighborhood $U$ of $p$ belongs to  $\mathscr G$ and hence meet, so the point $p$ of $L$ is a cluster point of  $\mathfrak F(C)$.
\end{proof}
\begin{prop}\label{closure-L}
Let  $(L, J)$ be a  locale equipped with a Grothendieck topology; let  $k\in L$  and let $A$ be a sieve on  $k$. The point $p$ of $L$ lies in the closure  of $A$ if and only if there is a $S$-filter  $\mathscr G $ such that  $A\in\mathscr G(k) $ and $\mathscr G(k) $ {\bf converges} to $p$.
\end{prop}
\begin{proof}
Let us begin by assuming that The point $p$ of $L$ lies in the closure  of $A$; from definition \ref{converges}, it follows that  every $\mathfrak{G}$-neighborhood $V$ of $p$ meets $A$. Then 
$
\mathscr B(k) =\{A\cap V\mid V \,\ \text{is a $\mathfrak{G}$-neighborhood of}\,\ p\}
$
is a base  for a $S$-filter  $\mathscr G$, in such a way that  $\mathscr G(k) $ {\bf converges} to $p$.

Conversely, if  $A \in \mathscr G(k)$ and  $\mathscr G(C) $ {\bf converges} to $p$ then  $p$ is a cluster point  of $\mathscr G(k) $ and hence $p$ lies in the closure  of $A$.
\end{proof}
 
\subsection{Compactness on Locales}
In this section we consider a concept of compactness on locales entirely different from the one described by P. J Johnstone in \cite{PJ1}.
\begin{defi}\label{l-compact}
.
 We shall say that an object $k\in L$ 
 
 \begin{itemize}
 \item Is {\bf  quasi-compact} if, for every $S$-filter $\mathfrak F$  on  $L$,  $\mathfrak F(k)$ has at least one cluster point.
 \item Is  {\bf Hausdorff} if, for every $S$-filter $\mathfrak F$  on  $L$,  $\mathfrak F(k)$ has no more that one limit point. 
 \item Is  {\bf compact} if, for every $S$-filter $\mathfrak F$  on  $L$,  $\mathfrak F(k)$ is quasi-compact and Hausdorff.
 \end{itemize}
\end{defi}
\begin{lemma}\label{ul-comp}
Let  $(L, J)$ be a  locale equipped with a Grothendieck topology. An object $k\in L$      is compact if and only if, for every $S$-ultrafilter $\mathfrak U$  on  $L$,  $\mathfrak U(k)$ is convergent.
\end{lemma}
\begin{proof}
First suppose that $\mathfrak F$  is an $S$-filter on  $L$. Proposition \ref{L-finer}
ensures that, for  every $S$-filter, there exists an $S$-ultrafilter $\mathfrak U$  finer than $\mathfrak F$, such that $\mathfrak U(k)$ converges to a point $p$ on $L$, therefore $p$ is a cluster point of $\mathfrak F(k)$.

Conversely, if, for an $S$-ultrafilter $\mathfrak U$, $\mathfrak U(k)$ has a cluster point then it converges to this point.
\end{proof}
\begin{prop}(Tychonoff)
 Let  $(L, J)$ be a  locale equipped with a Grothendieck topology. Every meet  (greatest lower bound) of compact objects in the  category $L$ is compact. 
\end{prop}
\begin{proof}
Suppose we have  chosen a collection $(k_{i})_{i\in I}$  of compact objects in $L$; equivalently, for every family $( \mathfrak U_{i} )_{i\in I}$ of ultrafilters  on   $L$, $\mathfrak U_i(k_i)$ is convergent.
Then the function $\mathfrak B$ which assigns to  each object\linebreak $k=\bigwedge_{i\in I} k_{i}$, the collection of sieves $ \mathfrak B(k)=\Big\{\displaystyle \bigcap_{i\in I}S_i \mid S_i\in \mathfrak F_{i}(k)\Big\}$, where $S_i=t_{k_i}$ is the maximal sieve on $k_i$  except for a finite number of indices, is basis of an ultrafilter on   $L$.
\end{proof}
\begin{ex}
Let $L_{\mathbb R}=\Omega(\mathbb R)$ be the locale of open subsets of the usual topolgy on the set  $\mathbb R$ of real numbers. Let $k_{1}$ be the open interval $(-1,1)\subset \mathbb R $, and let $p: \Omega(\mathbb R)\rightarrow 2$ be a point of $\Omega(\mathbb R)$ such that $p^{-1}(0)$ is the ideal $!(-2,2)$.  
\begin{enumerate}
\item Let $J$ be a Grothendieck topology  on $L_{\mathbb R}$ defined by

$J(k)= \begin{cases}
\big\{!k_{1},\{ (\frac{-1}{r}\,\frac{1}{r})\mid r\in \mathbb Z^{+} \}\big\} \hspace{3.1cm} \text{if}\,\  k= k_{1}\\

\Big\{!k,\big\{!k_{1},\{ (\frac{-1}{r}\,\frac{1}{r})\mid r\in \mathbb Z^{+} \}\big\}\cap (!k)\Big\}\hspace{1cm} \text {if $k\subseteq k_{1}$}\\

\Big\{!k,\{ (\frac{-1}{r}\,\frac{1}{r})\mid r\in \mathbb Z^{+}\}\cup\{\alpha,\beta)\} \  \text{if}\  k_{1}\subseteq(\alpha,\beta)\subseteq k\\

\{!k\} \hspace{6.5cm} \text{if}\,\ k_{1}\nsubseteq k.
\end{cases}
$

\item Let  $\mathcal N$ be a  cover-neighborhood of $(L_{\mathbb R}, J)$ and let $\mathfrak F$ be $S$-filter on  $L_{\mathbb R}$ so that   $\mathcal N_{p}(k_{1})= J(k_{1})\subseteq \mathfrak F(k_{1})$.
\item We may  conclude from propositions \ref{neighbor}, \ref{converges-L} and \ref{closure-L} that $k_{1}=(-1,1)\subset \mathbb R $ is a compact object of $\Omega(\mathbb R)$.
\end{enumerate}
\end{ex}


\begin{thebibliography}{10}
\bibitem{MGV} M. Artin, A. Grothendieck, J.-L. Verdier, \emph{Théorie des Topos et Cohomologie étale des Schémas, Tome 1, Théorie des Topos}, Séminaire de Géométrie Algébrique du Bois-Marie 1963-1964 (SGA 4).
\bibitem{NB} N. Bourbaki, \emph{ {\scriptsize Elements of Mathematics}, General Topology, Part I}, Addison-Wesley Pub. Co., 1966.
\bibitem{HC} H. Cartan, \emph{Filtres et ultrafiltres}, C. R. Acad. Sci. Paris, 1937.
\bibitem{PJ1}  P. T. Johnstone, \emph{ Sketches of an Elephant. A Topos Theory Compendium}, Two volumes, Oxford University Press, Oxford, 2002.
\bibitem{PJ2} P. T. Johnstone, \emph{ Stone spaces}, Cambridge University Press, London, 1982.
\bibitem{KM} P. Kostyrko et alt. \emph{ I-convergence and extremal I-limit points}, Math. Slov., 2005.
\bibitem{AL}{\sc A.J. Lindenhovius}, {\it Grothendieck topologies on posets}, arXiv:1405.4408v2, 2014.
\bibitem{SM} S. MacLane, \emph{ Categories for the Working Mathematician}, Springer-Verlag, New York / Heidelberg / Berlin, 1971.
\bibitem{MM}  S. MacLane and  I. Moerdijk, \emph{ Sheaves in Geometry and Logic,{ \scriptsize A first introduction to Topos theory}}, Springer-Verlag, New York / Heidelberg / Berlin, 1992.
\bibitem{NR} F. Nuray and W. H. Ruckle, \emph{Generalized statistical convergence and convergence free spaces}, J. Math. Anal.
Appl., 2000
\end{thebibliography}
\end{document}